\documentclass[12pt]{amsart}
\usepackage{enumerate}
\usepackage{amsmath, amssymb, amsthm, amsfonts}
\usepackage{tikz}
\usepackage{geometry}
\usepackage{bbm} %for the 1 symbol of indicator functions
\usepackage{url}

\newtheorem{introthm}{Theorem}
\newtheorem{introcor}[introthm]{Corollary}
\newtheorem{theorem}{Theorem}[section]
\newtheorem{proposition}[theorem]{Proposition}

\newtheorem{lemma}[theorem]{Lemma}

\newtheorem*{claim*}{Claim}

\theoremstyle{definition}

\newtheorem*{remark*}{Remark}

\newtheorem{definition}[theorem]{Definition}

\def\E{\mathbb{E}}

\def\IR{\mathbb{R}}
\def\IZ{\mathbb{Z}}

\def\eps{\varepsilon}

\def\ga{\gamma}
\def\Ga{\Gamma}

\def\si{\sigma}

\DeclareMathOperator{\Aut}{Aut}

\DeclareMathOperator{\dist}{dist}

\newcommand{\defeq}{\mathrel{\vcenter{\baselineskip0.5ex \lineskiplimit0pt
                     \hbox{\scriptsize.}\hbox{\scriptsize.}}}%
                     =}

\def\vertex{
\begin{tikzpicture}[baseline=-0.6ex]
\filldraw 
(0,0) circle (1.1pt);
\end{tikzpicture}
}

\def\edge{
\begin{tikzpicture}[baseline=-0.6ex]
\filldraw 
(0,-0.1) circle (1.1pt) -- (0,0.1) circle (1.1pt);
\end{tikzpicture}
}

\def\pthree{
\begin{tikzpicture}[baseline=-0.6ex]
\filldraw 
(0,0) circle (1.1pt) -- (0.2,0) circle (1.1pt) -- (0.4,0) circle (1.1pt) -- (0.6,0) circle (1.1pt);
\end{tikzpicture}
}

\def\distthree{
\begin{tikzpicture}[baseline=-0.6ex]
\filldraw (0,0) circle (1.1pt);
\draw (0.2,0) circle (1.1pt);
\draw (0.4,0) circle (1.1pt);
\filldraw (0.6,0) circle (1.1pt);
\draw[shorten <= 1.1pt, shorten >= 1.1pt] (0,0) -- (0.2,0);
\draw[shorten <= 1.1pt, shorten >= 1.1pt] (0.2,0) -- (0.4,0);
\draw[shorten <= 1.1pt, shorten >= 1.1pt] (0.4,0) -- (0.6,0);
\end{tikzpicture}
}

\def\sta{
\begin{tikzpicture}[baseline=-0.6ex]
\filldraw (0,0) circle (1.1pt) -- (0.2,0) circle (1.1pt) node[right, inner xsep=0pt, outer xsep=4pt]{\scriptsize{$d$}};
\filldraw (0,0) circle (1.1pt) -- (0.2,0.2) circle (1.1pt);
\filldraw (0,0) circle (1.1pt) -- (0.2,-0.2) circle (1.1pt); 
\end{tikzpicture}
}

\def\smallsta{
\begin{tikzpicture}[baseline=-0.4ex,scale=0.6, every node/.style={scale=0.6}]
\filldraw (0,0) circle (1.1pt) -- (0.2,0) circle (1.1pt) node[right, inner xsep=0pt, outer xsep=4pt]{\scriptsize{$d$}};
\filldraw (0,0) circle (1.1pt) -- (0.2,0.2) circle (1.1pt);
\filldraw (0,0) circle (1.1pt) -- (0.2,-0.2) circle (1.1pt); 
\end{tikzpicture}
}

\def\flower{
\begin{tikzpicture}[baseline=-0.6ex]
\draw (0,0) circle (1.1pt);
\filldraw (0.2,0) circle (1.1pt) node[right, inner xsep=0pt, outer xsep=4pt]{\scriptsize{$d$}};
\filldraw (0.2,0.2) circle (1.1pt);
\filldraw (0.2,-0.2) circle (1.1pt);
\draw[shorten <=1.1pt] (0,0) -- (0.2,0); 
\draw[shorten <=1.1pt] (0,0) -- (0.2,0.2); 
\draw[shorten <=1.1pt] (0,0) -- (0.2,-0.2); 
\end{tikzpicture}
}

\def\floweri{
\begin{tikzpicture}[baseline=-0.6ex]
\draw (0,0) circle (1.1pt);
\filldraw (0.2,0) circle (1.1pt) node[right, inner xsep=0pt, outer xsep=4pt]{\scriptsize{$i$}};
\filldraw (0.2,0.2) circle (1.1pt);
\filldraw (0.2,-0.2) circle (1.1pt);
\draw[shorten <=1.1pt] (0,0) -- (0.2,0); 
\draw[shorten <=1.1pt] (0,0) -- (0.2,0.2); 
\draw[shorten <=1.1pt] (0,0) -- (0.2,-0.2); 
\end{tikzpicture}
}

\def\floweriplus{
\begin{tikzpicture}[baseline=-0.6ex]
\draw (0,0) circle (1.1pt);
\filldraw (0.2,0) circle (1.1pt) node[right, inner xsep=0pt, outer xsep=4pt]{\scriptsize{$i+1$}};
\filldraw (0.2,0.2) circle (1.1pt);
\filldraw (0.2,-0.2) circle (1.1pt);
\draw[shorten <=1.1pt] (0,0) -- (0.2,0); 
\draw[shorten <=1.1pt] (0,0) -- (0.2,0.2); 
\draw[shorten <=1.1pt] (0,0) -- (0.2,-0.2); 
\end{tikzpicture}
}

\begin{document}

\title{Entropy inequalities for factors of IID}

\author[Backhausz]{\'{A}gnes Backhausz}
\address{ELTE E\"otv\"os Lor\'and University, Budapest, Hungary; Department of Probability Theory and Statistics; 
H-1117 Budapest, P\'azm\'any P\'eter s\'et\'any 1/c; 
and MTA Alfr\'ed R\'enyi Institute of Mathematics 
H-1053 Budapest, Re\'altanoda utca 13-15}
\email{agnes@math.elte.hu}

\author[Gerencs\'{e}r]{Bal\'{a}zs Gerencs\'{e}r}
\address{MTA Alfr\'ed R\'enyi Institute of Mathematics 
H-1053 Budapest, Re\'altanoda utca 13-15;
and ELTE E\"otv\"os Lor\'and University, Budapest, Hungary; Department of Probability Theory and Statistics 
H-1117 Budapest, P\'azm\'any P\'eter s\'et\'any 1/c}
\email{gerencser.balazs@renyi.mta.hu}

\author[Harangi]{Viktor Harangi}
\address{MTA Alfr\'ed R\'enyi Institute of Mathematics 
H-1053 Budapest, Re\'altanoda utca 13-15}
\email{harangi@renyi.hu}

\thanks{The first author was supported by 
the MTA R\'enyi Institute ``Lend\"ulet'' Limits of Structures Research Group 
and by the ``Bolyai \"Oszt\"ond\'ij" grant of the Hungarian Academy of Sciences.
The second author was supported by 
NKFIH (National Research, Development and Innovation Office) grant PD 121107. 
The third author was supported 
by Marie Sk{\l}odowska-Curie Individual Fellowship grant no.\ 661025 
and the MTA R\'enyi Institute ``Lend\"ulet'' Groups and Graphs Research Group.}


% valami (random) liftes grafos???
% 37A35 Dynamical systems and ergodic theory: Entropy and other invariants, isomorphism, classification
% 60K35 Probability: Interacting random processes; statistical mechanics type models; percolation theory
% 37A50 Dynamical systems and ergodic theory: Relations with probability theory and stochastic processes
% 05E18 Group actions on combinatorial structures

%\date{}

%
\begin{abstract}
This paper is concerned with certain invariant random processes 
(called factors of IID) on infinite trees. 
Given such a process, one can assign entropies to different finite subgraphs of the tree. 
There are linear inequalities between these entropies 
that hold for any factor of IID process 
(e.g.~``edge versus vertex'' or ``star versus edge''). 
These inequalities turned out to be very useful: 
they have several applications already, 
the most recent one is the Backhausz--Szegedy result 
on the eigenvectors of random regular graphs. 

We present new entropy inequalities in this paper. 
In fact, our approach provides a general ``recipe'' 
for how to find and prove such inequalities. 
Our key tool is a generalization of the edge-vertex inequality 
for a broader class of factor processes with fewer symmetries. 
\end{abstract}

\maketitle

%%%%%%%%%%%%%%%%%%%%%%%%%%%%%%%%%%%%%%%%%%%%%%%%%%%%%%%%%%%%%%%%%%%%%%%%%%%%%%%%%%%%%%%%%%%%%
%%                                                                                         %%
%%        INTRODUCTION                                                                     %%
%%                                                                                         %%
%%%%%%%%%%%%%%%%%%%%%%%%%%%%%%%%%%%%%%%%%%%%%%%%%%%%%%%%%%%%%%%%%%%%%%%%%%%%%%%%%%%%%%%%%%%%%

\section{Introduction}

\subsection{Entropy inequalities for processes on $T_d$}

For an integer $d \geq 3$ let $T_d$ denote the $d$-regular tree:
the (infinite) connected graph with no cycles 
and with each vertex having exactly $d$ neighbors. 

The main focus of this paper is the class of \emph{factor of IID processes}. 
Loosely speaking, independent and identically distributed
(say uniform $[0,1]$) random labels are assigned to the vertices of $T_d$,
then each vertex gets another label 
(a \emph{state} chosen from a finite \emph{state space} $M$) 
that depends on the labeled rooted graph as seen from that vertex,
all vertices ``using the same rule''. 
This way we get a probability distribution on $M^{V(T_d)}$ (called a factor of IID) 
that is invariant under the \emph{automorphism group} $\Aut(T_d)$ of $T_d$. 
A formal definition will be given in Section \ref{sec:1_2} below. 

One of the reasons why factor of IID processes have attracted a growing attention
in recent years is that they give rise to randomized local algorithms 
that can be carried out on arbitrary regular graphs with ``large essential girth'',
e.g.\ random regular graphs. See \cite{endreuj, E, V, hoppen} how 
factors of IID/local algorithms can be used to obtain 
large independent sets for large-girth graphs. 
Factors of IID are also studied by ergodic theory 
under the name of \emph{factors of Bernoulli shifts}, 
see Section \ref{sec:dynamical_systems} for details. 

The starting point of our investigations is 
the following \emph{edge-vertex entropy inequality} 
that holds for any factor of IID process on $T_d$: 
\begin{equation} \label{eq:original}
\frac{d}{2} H( \edge) \geq (d-1) H( \vertex) .
\end{equation}
Here $\vertex$ represents a vertex, and $H(\vertex)$ is the (Shannon) entropy 
of the (random) state of a vertex. 
Similarly, $\edge$ represents an edge, 
and $H(\edge)$ stands for the entropy of 
the joint distribution of the states of two neighbors. 
(Note that the state space $M$ is assumed to be finite here.) 
This inequality can be found implicitly in Lewis Bowen's work from 2009 \cite{bowen}. 
Rahman and Vir\'ag proved it in a special setting \cite{mustazeebalint}. 
A full and concise proof was given by Backhausz and Szegedy 
in \cite{invtree}; see also \cite{mustazee}. 
The counting argument behind this inequality 
actually goes back to a result of Bollob\'as 
on the independence ratio of random regular graphs \cite{bollobas}. 

A \emph{star-edge entropy inequality} was also proved in \cite{invtree}:
\begin{equation} \label{eq:star-edge}
H\left( \sta \right) \geq \frac{d}{2} H\left( \edge \right) ,
\end{equation}
where $H(\sta)$ denotes the entropy of the joint distribution 
of the states of a vertex and its $d$ neighbors. 
(Note that because of the $\Aut(T_d)$-invariance 
the distribution of every vertex/edge/star is the same.) 
%In fact, these inequalities are true 
%for a broader class of processes called \emph{typical} processes. 

The above inequalities played a central role 
in a couple of intriguing results recently: 
the Rahman--Vir\'ag result \cite{mustazeebalint} 
about the maximal size of a factor of IID independent set on $T_d$
and the Backhausz--Szegedy result \cite{ev}
on the ``local statistics'' of eigenvectors of random regular graphs. 

The goal of this paper is to obtain further inequalities 
between the entropies corresponding to different subgraphs of $T_d$. 
The ultimate goal would be to somehow describe the class 
of (linear) entropy inequalities that hold for any factor of IID process. 
We make progress towards this goal in this paper 
by developing a general method that can be 
used to find and prove such inequalities. 
See Section \ref{sec:1_3} for some examples 
of the new inequalities that this method produces. 
These examples include an upper bound for the (normalized) mutual information 
of two vertices at distance $k$. 
Another inequality we obtain is $H(\flower) \geq (d-1) H(\vertex)$, 
where $\flower$ represents the $d$ neighbors of any given vertex in $T_d$. 
This inequality can be used to improve earlier results 
about \emph{tree-indexed Markov chains}, 
see Section \ref{sec:Markov} for details. 

\subsection{General edge-vertex entropy inequalities} 
\label{sec:1_2}
Our key tool is a generalization of 
the edge-vertex inequality \eqref{eq:original} 
for processes with weaker invariance properties. 
For a given finite connected simple graph $G$ (that is not a tree itself) 
the universal cover is an infinite ``periodic'' tree $T$. 
Let $\Ga$ be a subgroup of the automorphism group $\Aut(T)$. 
By a \emph{$\Ga$-invariant process} over (the vertex set $V(T)$ of) $T$ 
we mean a probability distribution on $M^{V(T)}$ 
that is invariant under the natural $\Ga$-action. 
Although this makes sense for any measurable space $M$, 
in this paper the \emph{state space} $M$ will always be 
a finite set (with the discrete $\si$-algebra). 

Now we define factors of IID in this more general setting. 
A measurable function $F \colon [0,1]^{V(T)} \to M^{V(T)}$ 
is said to be a \emph{$\Ga$-factor} if it is $\Ga$-equivariant, that is,
it commutes with the natural $\Ga$-actions. 
Given an IID process $Z = \left( Z_v \right)_{v \in V(T)}$
on $[0,1]^{V(T)}$, applying $F$ yields a factor of IID process $X = F(Z)$,
which can be viewed as a collection $X = \left( X_v \right)_{v \in V(T)}$ of
$M$-valued random variables. It follows immediately %from the definition
that the distribution of $X$ is indeed $\Ga$-invariant. 

In the special case when the degree of each vertex of $G$ is the same 
(that is, when $G$ is \emph{$d$-regular} for some $d$) 
the universal cover is the $d$-regular tree $T_d$. 
If we simply say \emph{factor of IID process on $T_d$} 
(without specifying the group $\Ga$), 
we usually refer to the case when $\Ga$ is the full automorphism group $\Aut(T_d)$. 
The edge-vertex inequality \eqref{eq:original} holds 
for any $\Aut(T_d)$-factor of IID process. 
The next theorem and its corollary provide generalizations of \eqref{eq:original} 
for certain subgroups $\Ga$ of $\Aut(T_d)$. 
\begin{introthm} \label{thm:general}
Suppose that $G$ is a finite connected (simple) graph, $T$ is the universal cover of $G$, 
$\varphi \colon T \to G$ is an arbitrary fixed covering map. 
By $\Ga_\varphi \leq \Aut(T)$ we denote the group of \emph{covering transformations}, 
that is, the automorphisms $\ga \in \Aut(T)$ for which $\varphi \circ \ga = \varphi$. 
Let $M$ be a finite state space and $X$ a $\Ga_\varphi$-factor of IID process on $M^{V(T)}$. 
Given a vertex $v$ of the base graph $G$ let $\mu^X_v$ denote 
the distribution of $X_{\hat{v}}$ for any lift $\hat{v}$ of $v$. 
Similarly, for an edge $e \in E(G)$ let $\mu^X_e$ be the joint distribution 
of $(X_{\hat{u}}, X_{\hat{v}})$ for any lift $\hat{e} = (\hat{u}, \hat{v})$ of $e$. 
Note that these distributions are well defined 
because of the $\Ga_\varphi$-invariance of the process. 
Then the Shannon entropies of these distributions satisfy the following inequality:
\begin{equation} \label{eq:general}
\sum_{e \in E(G)} H(\mu^X_e) \geq \sum_{v \in V(G)} (\deg v - 1) H(\mu^X_v) ,
\end{equation}
where $\deg v$ is the degree (i.e.~number of neighbors) of the vertex $v$ in $G$. 
\end{introthm}
Compare this with the trivial upper bound 
$ \sum_{e \in E(G)} H( \mu^X_e ) \leq \sum_{v \in V(G)} \deg(v) H( \mu^X_v )$, 
where we have equality if and only if the states of two neighbors are independent. 
Thus the above theorem can be considered as a quantitative result 
as to ``how independent'' neighboring states are in a factor of IID process. 

We state the special case when $G$ is $d$-regular in a separate corollary. 
\begin{introcor} \label{cor:regular_case} 
Let $\varphi \colon T_d \to G$ be a covering map 
for a finite $d$-regular connected (simple) graph $G$ with $d \geq 3$. 
Using the notations of Theorem \ref{thm:general}, 
for any $\Ga_\varphi$-factor of IID process on $M^{V(T_d)}$ it holds that 
\begin{equation} \label{eq:regular_case}
\sum_{e \in E(G)} H( \mu^X_e ) \geq (d - 1) \sum_{v \in V(G)} H( \mu^X_v ) .
\end{equation}
\end{introcor}
This essentially says that \eqref{eq:original} holds for $\Ga_\varphi$-factors 
if $H( \vertex)$ and $H( \edge)$ are replaced 
by the average of the entropies of different ``types'' of vertices/edges. 
(Note that the number of edges is equal to $d/2$ times the number of vertices.)
This means that the original edge-vertex entropy inequality \eqref{eq:original} 
for $\Aut(T_d)$-factors follows from \eqref{eq:regular_case} for any $d$-regular $G$. 
Indeed, given an $\Aut(T_d)$-factor, it is also a $\Ga_\varphi$-factor 
with the extra property that each vertex/edge has the same distribution.

Another special case of Corollary \ref{cor:regular_case} is a result of Lewis Bowen 
saying that the so-called \emph{$f$-invariant} is non-negative for factors of Bernoulli shifts,  
see Section \ref{sec:dynamical_systems}. 

We will prove Theorem \ref{thm:general} in Section \ref{sec:proof} 
by considering random finite lifts of the base graph $G$ 
and counting the (expected) number of $M$-colorings on these lifts with the property that 
the ``local statistics'' of the coloring is close to that of the process $X$. 

\subsection{New inequalities}
\label{sec:1_3}
As we have mentioned, if we apply Corollary \ref{cor:regular_case} 
to an $\Aut(T_d)$-factor, then we simply get the original version \eqref{eq:original}. 
Hence it appears, falsely, that these more general inequalitites cannot be used to obtain 
new results in the most-studied special case of $\Aut(T_d)$-factors. 

The point is that starting from an $\Aut(T_d)$-factor of IID process $Y$ on $T_d$, 
there are many ways to turn this into a $\Ga_\varphi$-factor $X$ 
because one can use the extra structure on $T_d$ 
given by a covering $\varphi \colon T_d \to G$. 
Then applying Corollary \ref{cor:regular_case} to this new process $X$ 
yields an inequality for the original process $Y$. 
We demonstrate this on the following simple example. 
Let $G=K_{d+1}$ be the complete graph on $d+1$ vertices which is clearly $d$-regular. 
Let $o$ denote a distinguished vertex of $G$. 
Given a $T_d \to G$ covering map $\varphi$, every vertex of $T_d$ is either a lift of $o$, 
or has a unique neighbor that is a lift of $o$ (see Figure \ref{fig1}). 
Suppose that $Y$ is an $\Aut(T_d)$-factor of IID on $T_d$, and set 
$$ X_{v} \defeq Y_{v'} 
\mbox{, where $v'$ is the unique vertex such that } \varphi(v')=o \mbox{ and } \dist(v,v') \leq 1. $$
\begin{figure}
\centering
\includegraphics[width=6cm]{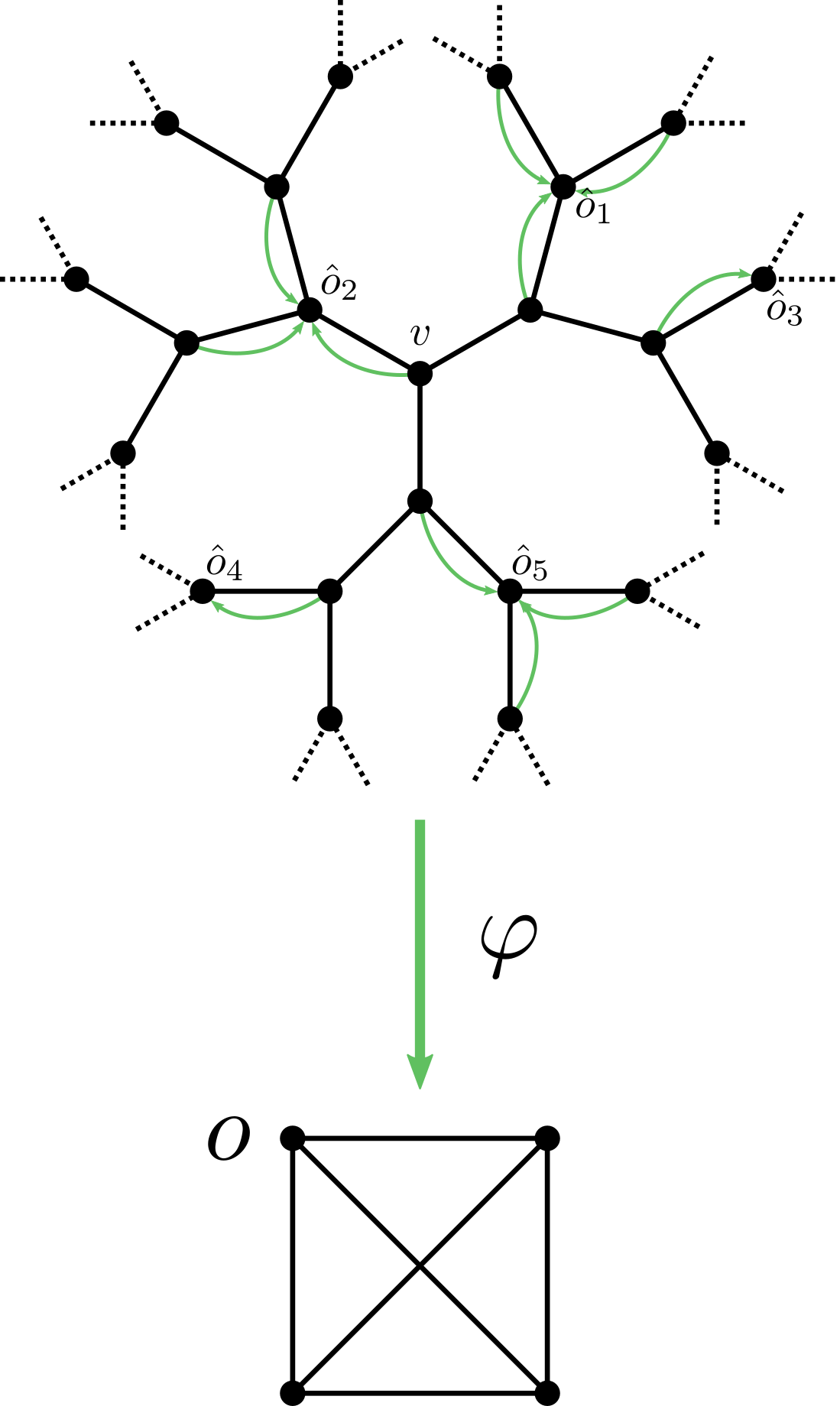}
\caption{The case $d=3$: $\hat{o}_i$ are the lifts of $o \in V(K_4)$; $X_v$ is defined as $Y_{\hat{o}_2}$.}
\label{fig1}
\end{figure}
It is easy to see that $X = \left( X_v \right)_{v \in V(T_d)}$ 
is a $\Ga_\varphi$-factor of IID and hence Corollary \ref{cor:regular_case} can be applied to $X$. 
Given two neighboring vertices $u$ and $v$ in $T_d$, 
the corresponding $u'$ and $v'$ either coincide 
(if $\varphi(u)=o$ or $\varphi(v)=o$), 
or they have distance $3$. 
It follows that 
$$ H(\mu^X_e) = \begin{cases}
H(\vertex) & \mbox{if the edge $e \in E(G)$ is incident to $o$,}\\
H(\distthree) & \mbox{otherwise,}
\end{cases} $$
where $\distthree$ represents two vertices of distance $3$, 
and the notations $H(\vertex)$ and $H(\distthree)$ refer to 
entropies corresponding to the ($\Aut(T_d)$-factor of IID) process $Y$.
Substituting these and $H(\mu^X_v) = H(\vertex)$ into \eqref{eq:regular_case} we obtain,
after cancellations, the following inequality for the process $Y$: 
$$ H(\distthree) \geq \left( 2 - \frac{2}{d(d-1)} \right) H(\vertex) .$$
This actually means that the normalized mutual information $I(Y_u; Y_v) / H(Y_v)$ 
is at most $\frac{2}{d(d-1)}$ for any vertices $u$ and $v$ of distance $3$ in $T_d$. 
The above argument can be generalized to obtain the following bounds 
for the normalized mutual information for arbitrary distance $\dist(u,v)=k$. 
A different proof for this result can be found in an earlier paper \cite{mut_inf} 
of the second and third author. 
\begin{introthm} \label{thm:mut_info}
\cite[Theorem 1]{mut_inf} 
Let $d\geq3$ be an integer. For any $u,v \in V(T_d)$ at distance $k$ 
and for any $\Aut(T_d)$-factor of IID process $Y$ on $T_d$ we have 
\begin{equation}
\frac{I(Y_u;Y_v)}{H(Y_v)} \leq \begin{cases}
\frac{2}{d(d-1)^l} & \mbox{ if $k=2l+1$ is odd,} \\
\frac{1}{(d-1)^l} & \mbox{ if $k=2l$ is even.}
\end{cases}
\end{equation}
\end{introthm}
Our general method is described in Section \ref{sec:3}, 
it provides countless new entropy inequalities. 
We list a few examples in the rest of the introduction. 

Let us fix an $\Aut(T_d)$-factor of IID process $Y$. 
Then for a finite set $V \subset V(T_d)$ 
the entropy of the joint distribution of $Y_v$, $v \in V$, 
will be denoted by $H(V)$. 
Because of the $\Aut(T_d)$-invariance of the process 
this joint distribution, and hence $H(V)$, depends only 
on the ``isomorphism type'' of $V$ in $T_d$. 

For instance, if $V$ consists of the four vertices of a path of length three, 
then we do not need to specify where this path is in $T_d$ 
and we can simply write $H(\pthree)$ for $H(V)$. 
The next theorem compares $H(\pthree)$ to $H(\edge)$. 
\begin{introthm} \label{thm:path-edge}
The following \emph{path-edge inequality} holds for any $\Aut(T_d)$-factor of IID process on $T_d$:
$$
%H(\ptwo) \geq \frac{3d-4}{2d-2} H(\edge) \mbox{ and } 
H(\pthree) \geq \frac{2d-3}{d-1} H(\edge) .$$
\end{introthm}
%
%If we combine \eqref{eq:original} and \eqref{eq:star-edge}, 
%then we get that $H(\sta) \geq (d-1) H(\vertex)$. 
%We show that this actually holds even if we remove 
%the central vertex from the star $\sta$: 
Another new inequality we obtain is 
\begin{equation} \label{eq:flower}
H(\flower) \geq (d-1) H(\vertex) .
\end{equation}
The following two theorems generalize this inequality in different ways.
\begin{introthm} \label{thm:sphere}
Let $S_k$ denote the set of vertices at distance $k$ from a fixed vertex of $T_d$. 
Then for any $\Aut(T_d)$-factor of IID process it holds that 
\begin{equation} \label{eq:sphere}
H(S_k) \geq (d-1)^k H(\vertex) .
\end{equation}
\end{introthm}
\begin{introthm} \label{thm:flower}
Let $\floweri$ denote the set of $i$ neighbors of a fixed vertex. 
Then for any $\Aut(T_d)$-factor of IID process and for any $1 \leq i < d$ it holds that 
$$ (d-i) H(\floweriplus) \geq (d-i-1) H(\floweri) + (d-1) H(\vertex),$$
and hence by induction for any $1\leq i \leq d$:
$$ H(\floweri) \geq \frac{id-2i+1}{d-1} H(\vertex), \quad  
\mbox{in particular, } 
H(\flower) \geq (d-1) H(\vertex) .$$
\end{introthm}
We will see in Section \ref{sec:sharpness} that 
each of these inequalities is sharp in the sense that 
there are $\Aut(T_d)$-factors of IID processes for which 
the two sides of the inequality are asymptotically equal. 
We will also examine how strong our new inequalities are: 
it turns out that \eqref{eq:flower} and \eqref{eq:sphere} 
are stronger than \eqref{eq:original} and \eqref{eq:star-edge} 
for Markov chains indexed by $T_d$.  

\subsection*{Outline of the paper}
The rest of the paper is structured as follows.
In Section \ref{sec:2} we go through basic definitions and 
elaborate on the strength of Theorem \ref{thm:general} for different base graphs. 
In Section \ref{sec:new_inequalities} we describe our general method for deriving 
new entropy inequalities from our general edge-vertex inequalities. 
In Section \ref{sec:sharpness} we show that these new inequalities are sharp, 
and we compare them to previously-known ones. 
Finally, the proof of Theorem \ref{thm:general} is given in Section \ref{sec:proof}. 

\subsection*{Acknowledgments}
We are grateful to B\'alint Vir\'ag and M\'at\'e Vizer for fruitful discussions on the topic. 

%%%%%%%%%%%%%%%%%%%%%%%%%%%%%%%%%%%%%%%%%%%%%%%%%%%%%%%%%%%%%%%%%%%%%%%%%
%%%%%%%%%%%%%%%%%%%%%%%%%%%%%%%%%%%%%%%%%%%%%%%%%%%%%%%%%%%%%%%%%%%%%%%%%
%%%%%%%%%%%%%%%%%%%%%%%%%%%%%%%%%%%%%%%%%%%%%%%%%%%%%%%%%%%%%%%%%%%%%%%%%

\section{Preliminaries} \label{sec:2}

\subsection{Factors of IID}
\label{sec:defs}

Suppose that a group $\Ga$ acts on a countable set $S$.
Then $\Ga$ also acts on the space $M^S$ for a set $M$:
for any function $f \colon S \to M$ and for any $\ga \in \Ga$ let
\begin{equation} \label{eq:action}
(\ga \cdot f)(s) \defeq f( \ga^{-1} \cdot s) \quad \forall s \in S .
\end{equation}
First we define the notion of factor maps.
\begin{definition}
Let $M_1,M_2$ be measurable spaces and
$S_1,S_2$ countable sets with a group $\Ga$ acting on both.
A measurable mapping $F \colon M_1^{S_1} \to M_2^{S_2}$ is
said to be a \emph{$\Ga$-factor} if it is $\Ga$-equivariant,
that is, it commutes with the $\Ga$-actions.
\end{definition}

By an \emph{invariant process} on $M^S$ we mean an $M^S$-valued random variable 
(or a collection of $M$-valued random variables) 
whose (joint) distribution is invariant under the $\Ga$-action. 
For example, if $Z_s$, $s\in S_1$, are independent and identically distributed 
$M_1$-valued random variables, then we say that 
$Z = \left( Z_s \right)_{s \in S_1}$ is an IID process on $M_1^{S_1}$. 
Given a $\Ga$-factor $F \colon M_1^{S_1} \to M_2^{S_2}$,
we say that $X \defeq F(Z)$ is a \emph{$\Ga$-factor of the IID process $Z$}.
It can be regarded as a collection of $M_2$-valued random variables: 
$X = \left( X_s \right)_{s \in S_2}$. 

The results of this paper are concerned with factor of IID processes on infinite trees $T$: 
$S_1$ and $S_2$ are the vertex set $V(T)$ and $\Ga$ is a subgroup of the automorphism group $\Aut(T)$. 
The most important special case is $T=T_d$ and $\Ga = \Aut(T_d)$. 
When we say $\Ga$-factor of IID process, 
we should also specify which IID process we have in mind
(that is, specify $M_1$ and a probability distribution on it). 
By default we will work with the uniform distribution on $[0,1]$. 
In fact, as far as the class of $\Aut(T_d)$-factors is concerned, 
it does not really matter which IID process we consider. 
For example, for the uniform distribution on $\{0,1\}$ 
we get the same class of factors as for the uniform distribution on $[0,1]$. 
This follows from the fact that these two IID processes 
are $\Aut(T_d)$-factors of each other \cite{karen_ball}. 

The other important special case is when $T$ is the universal cover 
of a finite connected simple graph $G$ and $\Ga=\Ga_\varphi$ is the group of 
covering transformations for a covering $\varphi \colon T \to G$. 
In this case it holds that for any $\hat{v}_1,\hat{v}_2 \in V(T)$ 
with $\varphi(\hat{v}_1) = \varphi(\hat{v}_2)$ there exists a unique $\ga \in \Ga_\varphi$ 
such that $\ga( \hat{v}_1 ) = \hat{v}_2$. 
It follows that if we choose a fixed pre-image $\bar{v} \in \varphi^{-1}(v)$ 
for every vertex $v \in V(G)$ of the base graph, 
then a $\Ga_\varphi$-factor $F \colon [0,1]^{V(T)} \to M^{V(T)}$ is determined by 
the functions $f_{\bar{v}} \defeq \pi_{\bar{v}} \circ F \colon [0,1]^{V(T)} \to M$, 
where $\pi_{\bar{v}}$ denotes the coordinate projection $M^{V(T)} \to M$ 
corresponding to the vertex $\bar{v}$. 
Conversely, any collection of measurable functions $f_{\bar{v}} \colon [0,1]^{V(T)} \to M$, 
$v \in V(G)$, gives rise to a $\Ga_\varphi$-factor mapping. 
(Note that an $\Aut(T_d)$-factor $F$ is determined 
by a single function $f_{o} \defeq \pi_{o} \circ F  \colon [0,1]^{V(T)} \to M$, 
but in that case $f_o$ needs to be invariant under 
all automorphisms of $T_d$ fixing the vertex $o \in V(T_d)$. 
See \cite[Section 2.1]{one_ended_tail} for details.)

\subsection{Finite-radius factors} 
\label{sec:block_factors}
Let $X$ be a $\Ga$-factor of the IID process $Z$. 
We say that $X$ is a \emph{finite-radius factor} (or a \emph{block factor}) 
if there exists a positive integer $R$ such that 
for any vertex $v$ the value of $X_v$ depends only on 
the values $Z_u$ for vertices $u$ in the $R$-neighborhood around $v$. 

Can a factor of IID process be approximated by finite-radius factors? 
In many cases the answer is positive. This means that 
it suffices to prove certain statements for finite-radius factors. 
For instance, in the proof of Theorem \ref{thm:general} we will need the fact 
that an arbitrary $\Ga_\varphi$-factor of IID process is 
the weak limit of finite-radius $\Ga_\varphi$-factors. 
As we have seen, a $\Ga_\varphi$-factor $F$ is determined 
by finitely many measurable $[0,1]^{V(T)} \to M$ maps. 
The pre-image of an element $m \in M$ under such a map 
is a measurable set in the product space $[0,1]^{V(T)}$,  
and, as such, it can be approximated 
by a finite union of measurable cylinder sets. 
Since $M$ is finite in our case, it follows that 
any measurable $[0,1]^{V(T)} \to M$ map can be approximated 
by maps for which all the pre-images are finite unions of cylinder sets, 
and consequently any $\Ga_\varphi$-factor can be approximated by finite-radius factors. 

\subsection{Finite coverings}
\label{sec:finite_coverings}
Theorem \ref{thm:general} provides an inequality 
for any finite base graph $G$. Next we elaborate on 
how these inequalities are related to each other. 

Suppose that $G_1$ and $G_2$ are finite connected (simple) graphs 
such that there is a covering map $\psi \colon G_2 \to G_1$. 
Then the $G_2$-version of Theorem \ref{thm:general} is 
stronger than the $G_1$-version. 
Indeed, let $T$ denote their universal cover. 
Given a covering map $\varphi_2 \colon T \to G_2$, 
setting $\varphi_1 \defeq \psi \circ \varphi_2$ 
yields a $T \to G_1$ covering map, see Figure \ref{fig2}. 
Clearly $\Ga_{\varphi_2} \leq \Ga_{\varphi_1}$. 
It follows that any $\Ga_{\varphi_1}$-factor of IID process $X$ on $T$ 
is also $\Ga_{\varphi_2}$-factor with the extra property that 
$\mu^X_v$ ($\mu^X_e$) depends only on the $\psi$-image of $v \in V(G_2)$ ($e \in E(G_2)$). 
Therefore it is easy to see that if we take 
the $G_2$-version of the general edge-vertex inequality \eqref{eq:general} 
and apply it to a $\Ga_{\varphi_1}$-factor, 
we simply get back the $G_1$-version of \eqref{eq:general}. 

This means that one can get stronger and stronger versions of \eqref{eq:general} 
by repeatedly lifting the finite base graph $G$.
\begin{figure}
\centering
\includegraphics[width=6cm]{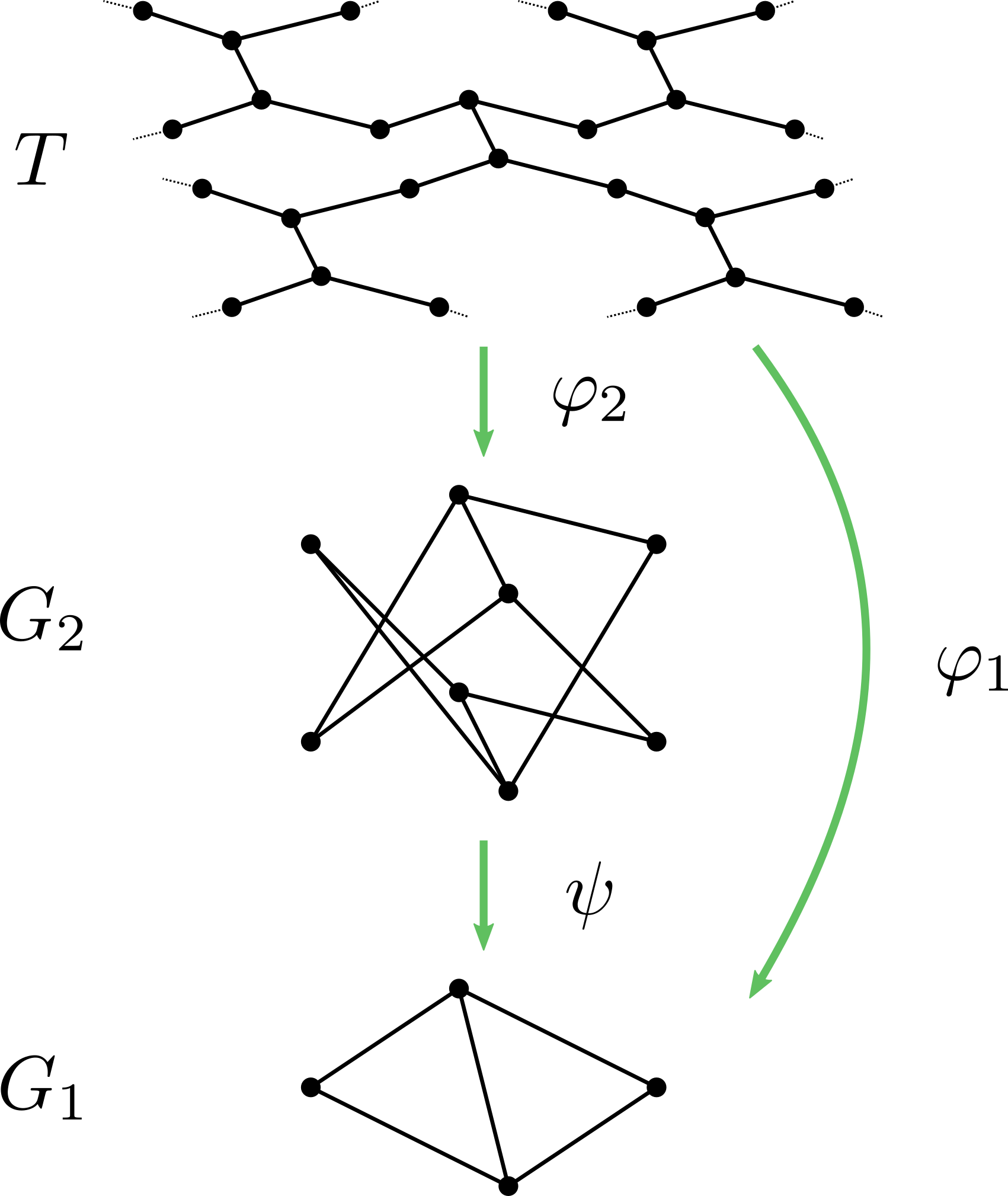}
\caption{Coverings $T \to G_2 \to G_1$}
\label{fig2}
\end{figure}

\subsection{Multiple edges and loops}
\label{sec:non-simple}
A graph is called \emph{simple} if it does not contain loops or multiple edges. 
For the sake of simplicity we stated (and we will prove) Theorem \ref{thm:general} 
for the case when the base graph $G$ is simple. 
What can be said for base graphs that are not simple? 

If $G$ has multiple edges (but no loops), 
then essentially the same result holds. 
The only difference is in the definition of a \emph{covering map} $T \to G$. 
In the case of simple graphs, one can simply say that a covering map is 
a mapping $V(T) \to V(G)$ such that the neighbors of a vertex $v$ 
are mapped bijectively to the neighbors of the image of $v$. 
When we have multiple edges, we also need to define the image of an edge: 
a covering map is a mapping $V(T) \to V(G)$ and a mapping $E(T) \to E(G)$ 
such that edges incident to a vertex $v$ are mapped bijectively 
to edges incident to the image of $v$. 
Once we know Theorem \ref{thm:general} for simple base graphs, 
it easily follows that it also holds when the base graph $G$ has multiple edges: 
simply take a finite simple graph $G_2$ that covers $G$; 
then the $G_2$-version of \eqref{eq:general} implies the $G$-version. 
(The proof of Theorem \ref{thm:general}
presented in Section \ref{sec:proof} would actually work
for base graphs with multiple edges.) 

As for loops the situation is a bit more complicated. 
In fact, one should distinguish between two kinds of loops. 
Loosely speaking: 
a \emph{full-loop} can be travelled in two directions 
(contributing to the degree of the vertex by $2$ 
and adding a free factor $\IZ$ to the fundamental group) 
while for a \emph{half-loop} there is just one way of ``going around'' 
(contributing to the degree by only $1$ 
and adding a free factor $\IZ_2$ to the fundamental group). 
For our purposes the difference between them is how they behave under coverings. 
In short, an edge ``double-covers'' a half-loop 
while two parallel edges are needed to double-cover a full-loop. 
We should define covering maps rigorously 
for graphs containing full-loops, half-loops, multiple edges. 
Then this could lead to a version of \eqref{eq:general} for arbitrary base graphs. 
The reason why we do not go into the details here is that, again, 
one can always take a finite simple lift of an arbitrary base graph 
and get a stronger version of the inequality. 

If $G$ has parallel edges 
(multiple edges between two vertices or more than one loops at one vertex), 
then we may choose not to ``distinguish'' some of those parallel edges 
but this would again lead to weaker inequalities. 
Note that in this terminology the original edge-vertex inequality \eqref{eq:original} 
would correspond to the case when the base graph $G$ 
consists of one vertex and $d$ undistinguished half-loops, 
which is the weakest version of \eqref{eq:regular_case} in the $d$-regular case. 

\subsection{Connections to dynamical systems}
\label{sec:dynamical_systems}
These processes can be viewed in the context of ergodic theory. 
An invariant process (as defined in Section \ref{sec:defs}) 
gives rise to a dynamical system over $\Ga$: 
the group $\Ga$ acts by measure-preserving transformations 
on the measurable space $M^S$ equipped with a probability measure 
(the distribution of the invariant process). 
An IID process simply corresponds to a (generalized) Bernoulli shift. 
Therefore factor of IID processes are factors of Bernoulli shifts. 

In fact, the general edge-vertex inequality \eqref{eq:general} 
is related to a result of Lewis Bowen 
saying that the so-called \emph{$f$-invariant} (for actions of the free group $F_r$) 
is non-negative for factors of the Bernoulli shift \cite[Corollary 1.8]{bowen}. 
This is essentially equivalent to Corollary \ref{cor:regular_case} 
in the special case when the base graph $G$ consists of one vertex 
and $r=d/2$ distinguished full-loops. 
See \cite[Section 2.3]{mut_inf} for details. 

%%%%%%%%%%%%%%%%%%%%%%%%%%%%%%%%%%%%%%%%%%%%%%%%%%%%%%%%%%%%%%%%%%%%%%%%%
%%%%%%%%%%%%%%%%%%%%%%%%%%%%%%%%%%%%%%%%%%%%%%%%%%%%%%%%%%%%%%%%%%%%%%%%%
%%%%%%%%%%%%%%%%%%%%%%%%%%%%%%%%%%%%%%%%%%%%%%%%%%%%%%%%%%%%%%%%%%%%%%%%%

\section{New inequalities for $\Aut(T_d)$-factors} \label{sec:3}
\label{sec:new_inequalities}
In the introduction we already demonstrated on a simple example 
how Corollary \ref{cor:regular_case} can be used 
to get new entropy inequalities for $\Aut(T_d)$-factors. 
In this section we describe our general method 
and present further examples.

Suppose that $Y$ is an $\Aut(T_d)$-factor of IID process on $M^{V(T_d)}$. 
Using the extra structure that a covering $\varphi \colon T_d \to G$ gives, 
$Y$ can be turned into a $\Ga_\varphi$-factor in many ways. 
For each $v \in V(G)$ we fix a non-backtracking walk starting at $v$. 
Then for any lift $\hat{v} \in V(T_d)$ of $v$ 
this walk can be lifted to get a path starting at $\hat{v}$. 
Let the endpoint of this path be assigned to $\hat{v}$. 
This assignment yields a mapping $f \colon V(T_d) \to V(T_d)$. 
It is easy to see that $f$ is $\Ga_\varphi$-equivariant, 
and consequently $X_u \defeq Y_{f(u)}$ defines a process $X$ 
that is a $\Ga_\varphi$-factor of IID, 
and hence Corollary \ref{cor:regular_case} can be applied to $X$. 
(The example in the introduction is the special case when $G = K_{d+1}$, 
and for the distinguished vertex $o \in V(G)$ we choose the walk $o$ of length $0$, 
and for any other vertex $v$ we choose the walk $v \to o$ of length $1$.) 

The general construction (where one can choose a finite collection of walks 
for each vertex) is described by the following lemma.
\begin{lemma}
Let $G$ be a finite connected $d$-regular (simple) graph 
and $\varphi \colon T_d \to G$ a covering map. 
Suppose that we have an $\Aut(T_d)$-factor of IID process $Y$ on $M^{V(T_d)}$. 
For any $v \in V(G)$ let us choose a finite collection of 
(non-backtracking) walks on $G$ (each starting at $v$): 
$W_{v,i}$, $1 \leq i \leq k_v$. 

For any lift $\hat{v} \in V(T_d)$ of $v$ 
we lift each $W_{v,i}$ starting at $\hat{v}$. 
Then we consider the endpoints of these $k_v$ paths and 
$X_{\hat{v}}$ is defined to be the $k_v$-tuple of the $Y$-labels of these endpoints. 
It can be seen easily that the obtained process $X$ 
is a $\Ga_\varphi$-factor of the IID process.
(Note that the state space for $X$ is 
$M' = M \cup (M \times M) \cup (M \times M \times M) \cup \ldots$) 
\end{lemma}
If we apply Corollary \ref{cor:regular_case} to this process $X$, 
then we will get an inequality between the entropies of 
various finite subsets of $V(T_d)$ for the original $\Aut(T_d)$-factor of IID process $Y$. 
This works for any choice of a finite $d$-regular base graph $G$ and walks $W_{v,i}$. 
In the remainder of this section we will show a few specific examples. 

To keep our notations simple, in this section 
we will write $\mu_v$ and $\mu_e$ for $\mu^X_v$ and $\mu^X_e$. 
Also, $H(\vertex)$ or $H(\edge)$, and more generally $H(V)$ for some $V \subset V(T_d)$, 
will always refer to the entropy corresponding to the original $\Aut(T_d)$-factor process $Y$. 

\subsection*{Two-vertex base graph, 
Theorem \ref{thm:path-edge} and \ref{thm:flower}}
As we discussed in Section \ref{sec:non-simple} 
the general edge-vertex inequality is true 
even when the base graph $G$ has multiple edges. 
So let $G$ be the graph with two vertices ($u$ and $v$) 
and $d$ multiple edges $e_1, \ldots, e_d$ between them. 
Given a positive integer $i \leq d-1$
the following $i$ walks (of length $1$) are associated to $u$: 
$u \stackrel{e_1}{\longrightarrow} v; \ldots; u \stackrel{e_i}{\longrightarrow} v$; 
while only the zero-length walk $v$ is associated to $v$. 
Then 
$$ H(\mu_u) = H(\floweri); \ H(\mu_v) = H(\vertex); \ 
H(\mu_{e_j}) = \begin{cases}
H(\floweri) & \mbox{if $j \leq i$,}\\
H(\floweriplus) & \mbox{if $j > i$.}
\end{cases} $$
Substituting these into \eqref{eq:regular_case} 
we get the first inequality in Theorem \ref{thm:flower}. 
The second inequality follows easily by induction. 

Next we consider the same base graph with different associated walks. 
Two walks starting at $u$, namely, $u$ and $u \stackrel{e_1}{\longrightarrow} v$; 
and two walks starting at $v$, namely, $v$ and $v \stackrel{e_1}{\longrightarrow} u$. 
It is easy to see that $H(\mu_u) = H(\mu_v)=H(\mu_{e_1}) = H(\edge)$, 
while for $j \geq 2$ we have $H(\mu_{e_j}) = H(\pthree)$, 
and consequently Theorem \ref{thm:path-edge} follows from \eqref{eq:regular_case}. 

\subsection*{Sphere versus vertex, Theorem \ref{thm:sphere}}
For a set $V \subset V(T_d)$ and a non-negative integer $k$ 
let $B_k(V) \defeq \left\{ u \ : \ \dist(u,V) \leq k \right\}$. 
The $k$-ball $B_k(\{o\})$ around some root $o$ will be denoted by $B_k$, 
while $S_k \defeq B_k \setminus B_{k-1} = \left\{ u \ : \ \dist(o,u)=k \right\}$ 
is the sphere of radius $k$. Our goal is to get an inequality 
between $H(S_k)$ and $H(\vertex)$. 

We will need the following auxiliary graph to define our base graph: 
let $T_{d,k}$ denote a finite tree that is isomorphic 
to the subgraph of $T_d$ induced by the $k$-ball $B_k$. 
The vertex set of $T_{d,k}$ can be partitioned into 
levels $0,1, \ldots, k$ (based on the distance to the root), 
level $i>0$ consisting of $d(d-1)^{i-1}$ vertices. 
All vertices have degree $d$ except vertices at level $k$ having degree $1$. 
Any vertex at level $0<i<k$ is connected to 
one vertex at level $i-1$ and $d-1$ vertices at level $i+1$. 

\begin{figure}
\centering
\includegraphics[width=6cm]{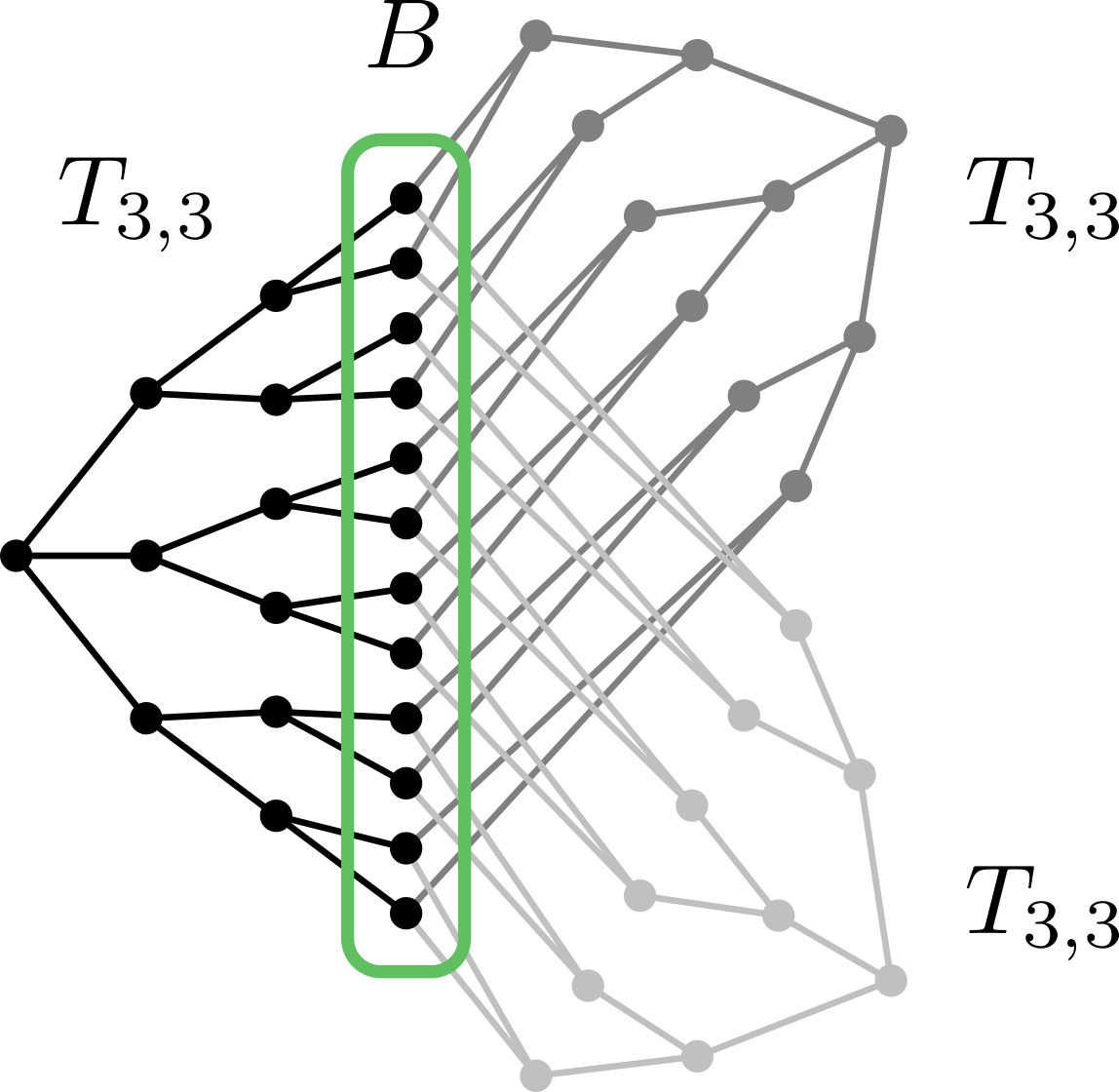}
\caption{Three copies of $T_{3,3}$ glued together along their boundary $B$}
\label{fig3}
\end{figure}
Now we take $d$ copies of $T_{d,k}$ and ``glue'' them along their level-$k$ vertices. 
This way we get a $d$-regular base graph $G$ 
(essentially $d$ balls of radius $k$ with a shared boundary). 
See Figure \ref{fig3} for the case $d=k=3$. 
The level-$k$ vertices 
(that is, vertices on the shared boundary 
that we will denote by $B$) 
only get the zero-length walks. 
Any other vertex $v$ belongs to exactly one copy of $T_{d,k}$. 
If we only use edges in this copy, 
then there is a unique path from $v$ to each vertex in $B$; 
let us associate these $|B|$ paths to $v$. Then we have 
$$ H(\mu_v) = \begin{cases}
H(\vertex) & \mbox{if $v \in B$,}\\
H(S_k) & \mbox{if $v \notin B$;} 
\end{cases} \mbox{ and } H(\mu_e)= H(S_k) \mbox{ for any } e \in E(G). $$ 
Using \eqref{eq:regular_case} we get that 
$$ d \left| E(T_{d,k}) \right| H(S_k) \geq 
\underbrace{(d-1) |B|}_{d(d-1)^k} H(\vertex) + 
d \underbrace{(d-1) \left( \left| V(T_{d,k}) \right| - |B| \right)}
_{| E(T_{d,k}) | - 1} H(S_k) ,$$ 
and Theorem \ref{thm:sphere} follows. 

\subsection*{Blow-ups}
By a \emph{blow-up} of an entropy inequality 
we mean the inequality we get if we replace each $H(V)$ with $H(B_k(V))$ 
for a fixed positive integer $k$. 
It is not hard to show that if a linear entropy inequality 
is true for all $\Aut(T_d)$-factors of IID, 
then the blow-ups of this inequality are 
also true for all $\Aut(T_d)$-factors of IID. 

For example, the blow-ups of the original edge-vertex inequality are: 
\begin{equation} \label{eq:ev_blowup}
\frac{d}{2} H\left( B_k( \edge ) \right) \geq (d-1) H\left( B_k( \vertex ) \right) .
\end{equation}
These blow-ups are closely related to Bowen's definition of the $f$-invariant \cite{f-invariant,bowen}; 
in particular, \eqref{eq:ev_blowup} follows from these papers. 

There is a very short proof for \eqref{eq:ev_blowup} 
using our general method: one can take any base graph $G$ and 
for each vertex take all non-backtracking random walks of length at most $k$. 
It is easy to see that every $H(\mu_v)$ equals $H\left( B_k( \vertex ) \right)$ 
and every $H(\mu_e)$ equals $H\left( B_k( \edge ) \right)$, 
and hence we get \eqref{eq:ev_blowup}. 
Moreover, if an inequality is attainable by our method, 
then so are its blow-ups: one needs to replace each associated walk in $G$ 
with all walks obtained by concatenating this walk and any walk of length at most $k$. 

We also mention that in \cite{ev} the blow-ups of the star-edge inequality \eqref{eq:star-edge} 
were proved for a broader class of invariant processes that were called \emph{typical processes}. 
These blow-up inequalities played a central role in the proof of the main result of that paper. 
(Loosely speaking, a process is typical if it arises 
as a limit of labelings of random $d$-regular graphs. 
Their significance lies in the fact that many questions about 
random regular graphs can be studied through typical processes.
It would be very interesting to know 
whether our new inequalities are also true for this broader class.) 

\subsection*{Mutual information decay, Theorem \ref{thm:mut_info}}
As we pointed out in the introduction, 
Theorem \ref{thm:mut_info} was already proved 
in an earlier paper \cite{mut_inf} of the second and third author. 
Next we show how this inequality follows easily from Corollary \ref{cor:regular_case}. 

We need to define the base graph $G$ 
slightly differently for odd and even $k$. 
For an odd distance $k=2l+1$ let us take 
two copies of $T_{d,l}$ and add edges between their boundaries 
(that is, their level-$l$ parts) in such a way 
that the obtained graph $G$ is $d$-regular. 
Figure \ref{fig4} shows the base graph for the case when $d=4; k=5; l=2$. 
\begin{figure}
\centering
\includegraphics[width=8cm]{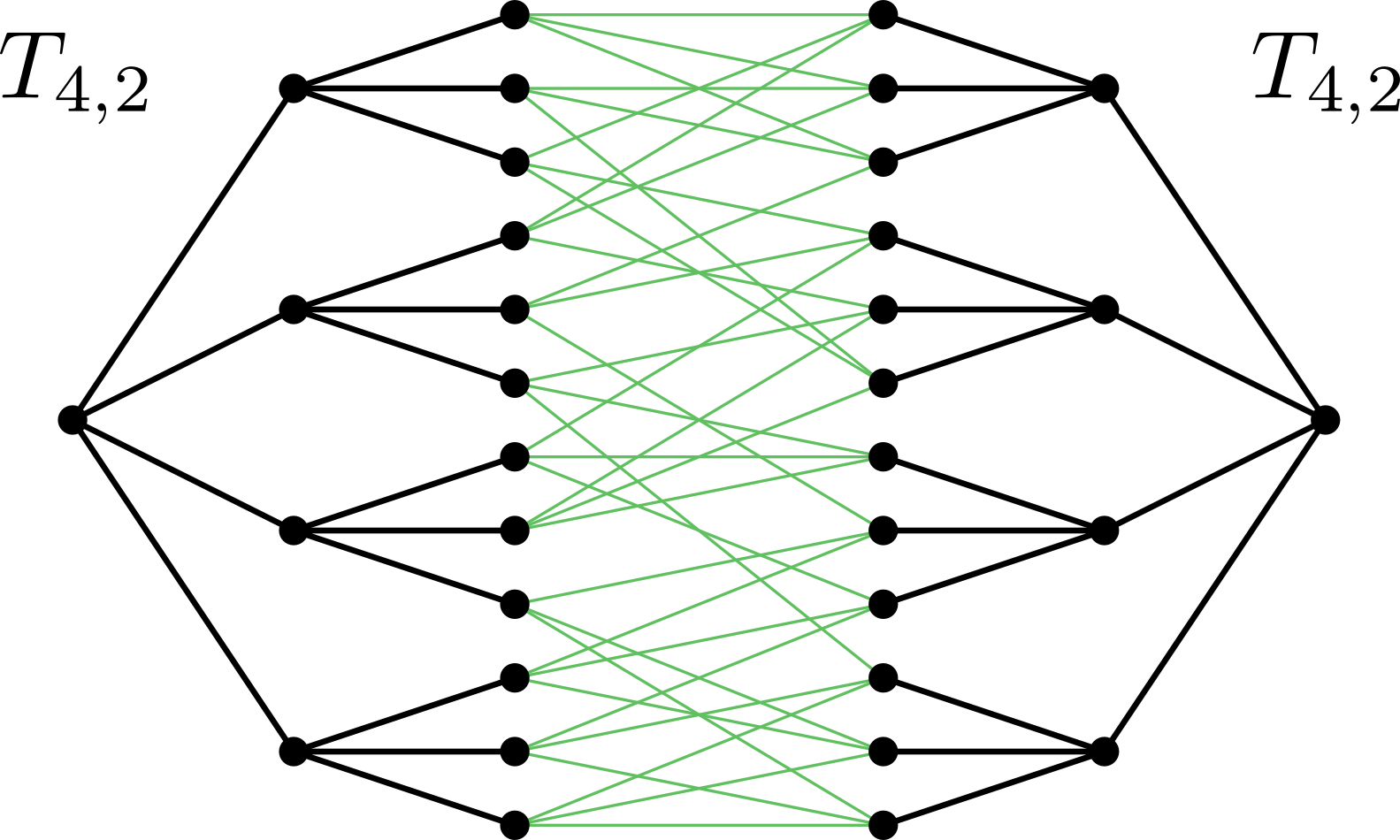}
\caption{Two disjoint copies of $T_{4,2}$ with additional edges between the boundaries}
\label{fig4}
\end{figure}

As for the case when $k=2l$ is even, 
one needs to connect the boundaries of a $T_{d,l}$ and a $T_{d,l-1}$. 
Their boundaries are not of the same size, though, 
so we need to take $d-1$ copies of $T_{d,l-1}$ and one copy of $T_{d,l}$. 
Then we can add edges connecting the boundary vertices of $T_{d,l}$ to 
the boundary vertices of the copies of $T_{d,l-1}$ in such a way that 
the obtained graph $G$ is $d$-regular. 

In both cases we have one walk associated to each vertex of $G$: 
the unique path going to the root inside that copy.
For all $v \in V(G)$ and for all original edges $e$ (going inside a copy) 
we have $H(\mu_v) = H(\mu_e) = H(\vertex)$. 
As for additional edges $e$ (going between the boundaries of different copies), 
$\mu_e$ is the joint distribution of $(Y_u,Y_v)$ 
for vertices $u,v$ at distance $k$. 
Substituting these into \eqref{eq:regular_case} 
leads to Theorem \ref{thm:mut_info}. 
The calculations are straightforward, we include the odd case $k=2l+1$ here. 
Let $B$ denote the boundary of $T_{d,l}$; then 
$$ 2 \left| E(T_{d,l}) \right| H(\vertex) 
+ (d-1) |B| H(Y_u,Y_v) \geq 2(d-1) \left| V(T_{d,l}) \right| H(\vertex) .$$ 
Then for the mutual information $I(Y_u;Y_v) \defeq 2 H(\vertex) - H(Y_u, Y_v)$ we have 
$$ \underbrace{(d-1)|B|}_{d(d-1)^l} \frac{ I(Y_u; Y_v) }{ H(\vertex) } 
\leq 2 \left| E(T_{d,l}) \right| + 2(d-1) |B| - 2(d-1) \left| V(T_{d,l}) \right| = 2 .$$

%%%%%%%%%%%%%%%%%%%%%%%%%%%%%%%%%%%%%%%%%%%%%%%%%%%%%%%%%%%%%%%%%%%%%%%%%
%%%%%%%%%%%%%%%%%%%%%%%%%%%%%%%%%%%%%%%%%%%%%%%%%%%%%%%%%%%%%%%%%%%%%%%%%
%%%%%%%%%%%%%%%%%%%%%%%%%%%%%%%%%%%%%%%%%%%%%%%%%%%%%%%%%%%%%%%%%%%%%%%%%

\section{Sharpness, comparisons, applications}
\label{sec:sharpness}

\subsection{Sharpness}
All the inequalities stated in this paper for $\Aut(T_d)$-factors 
(Theorem \ref{thm:mut_info}--\ref{thm:flower}) 
are sharp in the following sense. Given a linear entropy inequality 
it is natural to normalize it by dividing both sides by the entropy of a vertex. 
We claim that there exist $\Aut(T_d)$-factor of IID processes for which 
the two sides of the inequality are arbitrarily close to each other (after normalization). 
In fact, for each inequality the same examples can be used to demonstrate the sharpness. 
These examples were already presented in \cite{mut_inf} 
to show that the upper bound for the normalized mutual information is sharp. 
For the sake of completeness we briefly recall these examples. 

The idea is very simple: given IID labels at the vertices, 
let the factor process ``list'' all the labels 
within some large distance $R$ at any given vertex. 
One needs to be careful since listing 
the labels should be done in an $\Aut(T_d)$-invariant way. 
One possibility is to use the following lemma. 
\begin{lemma} \cite[Lemma 5.2]{mut_inf}
For any positive integer $L$ there exists a factor of IID coloring  
of the vertices of $T_d$ such that finitely many colors are used 
and vertices of the same color have distance greater than $L$. 
\end{lemma}
Let us fix $R$ and pick a very large $L$.
Let $C= \left( C_w \right)_{w \in V(T_d)}$ 
be a factor of IID coloring provided by the lemma above. 
Given a positive integer $N$ 
let $Z_w$, $w \in V(T_d)$ be IID uniform labels from $\{1,2,\ldots,N\}$. 
We set 
$$Y_v = \{(C_w,Z_w)~|~w \in B_R(v) \}.$$
Then $Y_v$ can be viewed as the list of
variables $(C_w, Z_w),~w\in B_R(v)$, ordered by $C_w$ (which
are all different if $L$ is large enough). This is now an $\Aut(T_d)$-invariant description.
Furthermore, conditioned on the coloring process $C$ 
the entropy corresponding to a finite subset $V \subset V(T_d)$ 
is $\left| B_R(V) \right| \log(N)$ provided that $L$ is large enough. 
On the other hand, the contribution of the coloring to the entropies 
does not depend on $N$, so it gets negligible as $N$ goes to infinity. 
One can easily check that if we replace $H(V)$ by $\left| B_R(V) \right|$ 
in any of our inequalities, then the two sides will be 
asymptotically equal as $R \to \infty$, and sharpness follows.

\subsection{Hierarchy of entropy inequalities}
We say that an entropy inequality $A$ is \emph{stronger} than an inequality $B$ 
($A \Rightarrow B$ in notation) if the following is true: 
whenever an $\Aut(T_d)$-invariant process $Y$ (not necessarily factor of IID) satisfies $A$, 
then $Y$ also satisfies $B$. 
There is a nested hierarchy between the blow-ups of the edge-vertex and star-edge inequalities: 
\begin{multline*}
\cdots \Rightarrow \frac{d}{2} H\left( B_{k+1}( \edge ) \right) 
\geq (d-1) H\left( B_{k+1}( \vertex ) \right) \Rightarrow
H\left( B_k( \sta ) \right) \geq \frac{d}{2} H\left( B_k( \edge ) \right) \\
\Rightarrow \frac{d}{2} H\left( B_k( \edge ) \right) 
\geq (d-1) H\left( B_k( \vertex ) \right) \Rightarrow \cdots 
\end{multline*}
In particular, the star-edge inequality \eqref{eq:star-edge} 
is stronger than the edge-vertex inequality \eqref{eq:original}, 
and, in turn, the blow-up \eqref{eq:ev_blowup} (for $k=1$) 
of the edge-vertex inequality implies the star-edge inequality. 

This can be seen using conditional entropies; 
we only include a sketch of the argument. 
For finite sets $U,W \subset V(T_d)$ 
let $H(W | U)$ denote the conditional entropy $H( Y_w, w \in W \ | \ Y_u, u \in U)$. 
We will only use this in the special case when $U \subset W$, 
where we have $H(W|U) = H(W)-H(U)$. 
 
To see that \eqref{eq:star-edge} is stronger than \eqref{eq:original}: 
for any invariant process $Y$ satisfying \eqref{eq:star-edge} we have 
$$ \frac{d}{2} H(\edge) \stackrel{\eqref{eq:star-edge}}{\leq} 
H(\sta) = H(\edge) + H(\sta | \ \edge) 
\leq H(\edge) + (d-1) H(\edge \ | \ \vertex) = d H(\edge) - (d-1) H(\vertex) ,$$
and \eqref{eq:original} follows. 

A similar argument shows that for a process satisfying \eqref{eq:ev_blowup} (for $k=1$) we have 
$$ \frac{2(d-1)}{d} \underbrace{ H( B_1(\vertex) ) }_{H(\smallsta)} 
\stackrel{\eqref{eq:ev_blowup}}{\leq} 
H( B_1(\edge) ) \leq H(\sta) + H(\sta | \ \edge) = 2 H(\sta) - H(\edge) ,$$ 
and \eqref{eq:star-edge} follows.

Similar arguments were known by Bowen in the dynamical system context, 
see \cite[Proposition 5.1]{f-invariant}.

\subsection{Tree-indexed Markov chains}
\label{sec:Markov}
We have already seen that all our new entropy inequalitites are sharp 
but the question remains: how strong are they compared to previously-known ones? 
Next we compare them for a specific class of processes.

An intriguing open problem about factor of IID processes 
is to determine the parameter regime where the \emph{Ising model} on $T_d$ 
can be obtained as a factor of IID process. 
More generally, given a Markov chain indexed by $T_d$ with some transition matrix, 
decide whether the corresponding invariant process is a factor of IID or not. 
(See \cite{lyons,invtree} and references therein.)

Here we focus on obtaining constraints for a Markov chain to be factor of IID. 
Two approaches have been used to show that a tree-indexed Markov chain cannot be factor of IID. 
The correlation bound given in \cite{cordec} 
implies that the spectral radius of the transition matrix 
is at most $1/\sqrt{d-1}$ in the factor of IID case. 
The edge-vertex entropy inequality yields another constraint. 
For the Ising model the former gives a slightly better result. 
There are examples, however, where the latter 
performs significantly better \cite[Theorem 5]{invtree}. 

One might think that the entropy approach can be improved 
by considering the stronger blow-up inequalities described above. 
However, for Markov chains all these blow-ups 
are equivalent to the edge-vertex inequality. 
This is due to the fact that for any connected subset $V \subset V(T_d)$ we have 
$$ H(V) =  H(\vertex) 
+ (|V|-1) \underbrace{ H( \edge \ | \ \vertex ) }_{H\left(\edge\right) - H(\vertex)} 
= (|V|-1) H(\edge) - (|V|-2) H(\vertex) $$ 
because of the Markov property. It follows that all known inequalities 
involving entropies of connected sets are equivalent to the edge-vertex inequality 
for tree-indexed Markov chains. In particular, $H(\sta) \geq (d-1) H(\vertex)$, 
which follows by combining \eqref{eq:original} and \eqref{eq:star-edge}, 
is also equivalent to \eqref{eq:original} for these processes.  

We claim that our new entropy inequalities \eqref{eq:sphere}, 
proved in Theorem \ref{thm:sphere} for $\Aut(T_d)$-factors of IID, 
are stronger than \eqref{eq:original} for tree-indexed Markov chains.  
\begin{proposition}
For tree-indexed Markov chains 
the inequality $H(S_k) \geq (d-1)^k H(\vertex)$ is 
stronger than the edge-vertex inequality \eqref{eq:original} 
and its blow-ups \eqref{eq:ev_blowup} for any given $k$. 
\end{proposition}
\begin{proof}
The inequality $H(S_k) \geq (d-1)^k H(\vertex)$ is clearly stronger 
than $H(B_k) \geq (d-1)^k H(\vertex)$. The latter, however, is equivalent to 
\eqref{eq:original} and \eqref{eq:ev_blowup} for tree-indexed Markov chains. 
\end{proof}
Therefore whenever the entropy approach performs better than the correlation bound, 
using Theorem \ref{thm:sphere} for any $k \geq 1$ instead of \eqref{eq:original} 
will give an even better result.

As for which $k$ we get the strongest inequality (for Markov chains), 
we do not have a complete answer. 
We can prove that for $k=2$ the theorem is stronger than for $k=1$, 
but we do not know if larger $k$ always provides stronger inequality in Theorem \ref{thm:sphere}. 

%$k=2$ is stronger than $k=1$ but we do not know if this is true for larger $k$

%%%%%%%%%%%%%%%%%%%%%%%%%%%%%%%%%%%%%%%%%%%%%%%%%%%%%%%%%%%%%%%%%%%%%%%%%
%%%%%%%%%%%%%%%%%%%%%%%%%%%%%%%%%%%%%%%%%%%%%%%%%%%%%%%%%%%%%%%%%%%%%%%%%
%%%%%%%%%%%%%%%%%%%%%%%%%%%%%%%%%%%%%%%%%%%%%%%%%%%%%%%%%%%%%%%%%%%%%%%%%

\section{Proof of the general edge-vertex inequality} 
\label{sec:proof}

To prove the original edge-vertex inequality \eqref{eq:original} 
one needs to count colorings with given ``local statistics'' 
on random $d$-regular graphs \cite{invtree,mustazee}. 
In order to obtain Theorem \ref{thm:general} we will generalize 
this argument for random lifts of a finite base graph $G$. 

Let us fix a finite connected simple graph $G$ and 
a covering map $\varphi \colon T \to G$ for the universal covering tree $T$. 
By $\Ga = \Ga_{\varphi}$ we denote the group of covering transformations of $T$. 
We will consider finite lifts $\hat{G}$ of $G$ and
colorings of the vertices of $\hat{G}$. 

\begin{definition} \label{def:stat} 
Let $\hat{G}$ be an $N$-fold lift of $G$. 
That is, we have a (deterministic) graph $\hat{G}$ 
and a covering $\hat{G} \to G$ such that 
every vertex/edge has exactly $N$ \emph{lifts} (i.e.~pre-images under the covering map). 
Suppose that $c \colon V(\hat{G}) \to M$ is 
a (deterministic) coloring for some finite set $M$ of colors. 

By the \emph{local statistics} of the coloring $c$ 
we mean the following distributions: 
given a vertex $v$ (or an edge $e$) of $G$, let $\mu_v^c$ (or $\mu_e^c$) be 
the ``empirical distribution'' of the colors of the $N$ lifts of $v$ (or $e$).
More precisely, for $v \in V(G)$ let $\mu_v^c$ be 
the distribution of $c( \hat{v} )$, 
where $\hat{v} \in V( \hat{G} )$ is chosen uniformly at random among the lifts of $v$. 
Similarly, for $e=(u,v) \in E(G)$ let $\mu_e^c$ denote 
the joint distribution of $\left( c( \hat{u} ) , c( \hat{v} ) \right)$, 
where $\hat{e} = (\hat{u}, \hat{v}) \in E( \hat{G} )$ is 
chosen uniformly at random among the lifts of $e$. 

Note that $\mu_e^c$ is a probability distribution on $M \times M$ 
with the two marginals being $\mu_u^c$ and $\mu_v^c$. 
Also, all the probabilities occuring in these distributions 
are multiples of $1/N$. 
\end{definition}

From this point on $\eps = \eps(N)$ will denote a positive quantity 
that slowly converges to $0$ as $N \to \infty$. 
To be more specific, let $\eps = C/ \log N$, 
where $C$ does not depend on $N$, but it might depend on the base graph $G$, 
the size of the state space $M$, and the radius $R$ of the factor process. 
Note that $C$ might be different at each occurence of $\eps$. 
The proof will have the following ingredients. 
(Some of the notions used here will be defined later.) 
\begin{enumerate}[a)]
%\item Every $\Ga$-factor of IID is the weak limit of finite-radius $\Ga$-factors, 
%and hence it is enough to prove the theorem for those. 
\item It holds with high probability that 
the random $N$-fold lift of a finite graph $G$ has large \emph{essential girth}, 
that is, the number of short cycles is small compared to the number of vertices. 
\item Given any finite-radius $\Ga$-factor of IID process $X$ 
with finite state space $M$ and a finite covering $\hat{G} \to G$ 
the following holds: there exists a deterministic $M$-coloring $c$ of $\hat{G}$ 
such that the local statistics $\mu_v^c$ and $\mu_e^c$ are $\eps$-close to 
$\mu_v^X$ and $\mu_e^X$ provided that the essential girth of $\hat{G}$ is large enough. 
\item Finally, we determine the expected number of $M$-colorings 
with given local statistics on a random $N$-fold lift of $G$. 
\end{enumerate}
The general edge-vertex inequality \eqref{eq:general} 
will follow easily by combining the above ingredients. 

\subsection*{a) Random lifts}

Given a finite simple base graph $G$ and a positive integer $N$, 
a \emph{random $N$-fold lift of $G$}, denoted by $\hat{G}_N$, 
is the following random graph: 
for each $v \in V(G)$ we take $N$ vertices 
$L_v \defeq \left\{ \hat{v}_1, \ldots, \hat{v}_N \right\}$, 
and for each $e = (u,v) \in E(G)$ we take a uniform random perfect matching 
between $L_u$ and $L_v$ (independently for every edge $e$). 
Figure \ref{fig5} shows such a random lift for a base graph with four vertices and five edges. 
\begin{figure}
\centering
\includegraphics[width=6cm]{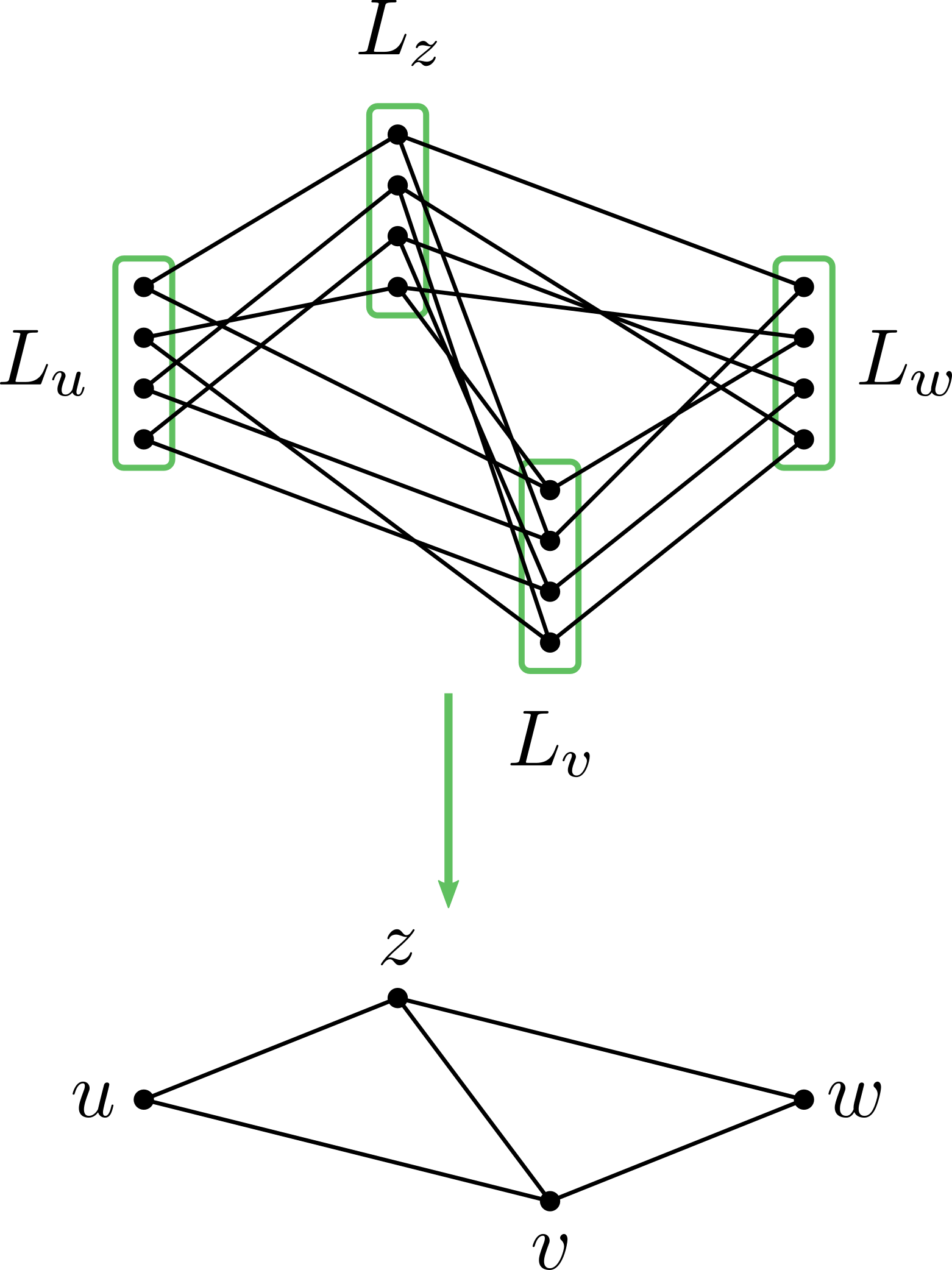}
\caption{The $4$-fold random lift of a finite simple graph.}
\label{fig5}
\end{figure}

The above definition works for base graphs without loops. 
In this paper we do not need to use the notion of random lift for base graphs with loops. 
Let us note nevertheless that random $d$-regular graphs 
can be considered as random lifts of the graph 
with one vertex and $d$ \emph{half-loops}. 

It is well known that a random $N$-fold lift has few short cycles. 
More precisely, \cite[Lemma 2.1]{fortin2012asymptotic} shows 
that for any fixed positive integer $l$ 
the expected number of $l$-cycles in a random $N$-fold lift 
stays bounded as $N \to \infty$. 
Using Markov's inequality this immediately implies that 
with high probability the number of cycles of length at most $l$ 
is small compared to the number of vertices, 
which, in turn, implies that the random lift 
is locally a tree around most vertices. 
The exact statement we will use is the following. 
\begin{lemma} \label{lem:large_girth}
Given any $G$ and any positive integer $R$ 
the random $N$-fold lift of $G$ has the following property 
with probability $1-o(1)$ as $N$ goes to infinity:  
the $R$-neighborhoods of all but at most $\eps N$ edges are trees. 
\end{lemma}

\subsection*{b) Projecting finite-radius factors onto large-girth graphs}
\label{sec:projecting}
The content of this section can be found in \cite[Section 2.1]{mustazee} 
for the $\Aut(T_d)$-invariant case. The following is a straightforward 
adaptation for our setting. 

Suppose that we have a finite-radius $\Ga$-factor of IID process with radius $R$ 
and let $F \colon [0,1]^{V(T)} \to M^{V(T)}$ be the corresponding $\Ga$-factor mapping. 
(See Section \ref{sec:defs} and \ref{sec:block_factors} for definitions.) 
Next we explain how one can ``project'' such a process onto finite lifts of $G$. 

Let $\hat{G}$ be a fixed (deterministic) lift of $G$. 
We call a vertex/edge of $\hat{G}$ \emph{$R$-nice} if its $R$-neighborhood is a tree. 
By the \emph{type} of a vertex $\hat{v} \in V(\hat{G})$ 
we mean its image $v \in V(G)$ under the covering map. 
Similarly, we can talk about the type of a vertex of the universal cover $T$. 
 
Given an $R$-nice vertex $\hat{v} \in V(\hat{G})$ 
and an arbitrary vertex $\bar{v} \in V(T)$ with the same type $v \in V(G)$, 
their $R$-neighborhoods are clearly isomorphic. 
Moreover, there is a unique isomorphism between these neighborhoods 
that preserves the vertex types. In what follows we will use this unique 
isomorphism to identify these neighborhoods. 

Now suppose that $[0,1]$ labels are assigned to the vertices of $\hat{G}$. 
We will refer to these labels as \emph{input labels}. 
Depending on these input labels 
we assign a state (i.e.~an element from $M$) to each vertex $\hat{v} \in V(\hat{G})$, 
that is, we define a $[0,1]^{V(\hat{G})} \times V(\hat{G}) \to M$ mapping. 
%or equivalently a $[0,1]^{V(\hat{G})} \to M^{V(\hat{G})}$ mapping. 
We pick an arbitrary fixed state $m_0 \in M$. 
If $\hat{v}$ is not $R$-nice, we assign $m_0$ to $\hat{v}$. 
If $\hat{v}$ is $R$-nice, then we can ``pretend'' that 
we are at a vertex $\bar{v}$ of the universal cover $T$: 
we copy the input labels onto the $R$-neighborhood of $\bar{v}$ 
and apply the function $f_{\bar{v}} \defeq \pi_{\bar{v}} \circ F \colon [0,1]^{V(T)} \to M$; 
the value of $f_{\bar{v}}$ gets assigned to $\hat{v}$. 
(Recall that $\pi_{\bar{v}}$ denotes the coordinate projection $M^{V(T)} \to M$ 
corresponding to the vertex $\bar{v}$.)

For any $\Ga$-factor process $X$ with finite radius $R$ 
and for any finite cover $\hat{G}$ of $G$ 
we described a mapping $[0,1]^{V(\hat{G})} \times V(\hat{G}) \to M$. 
If we choose the input labels randomly (IID and uniform $[0,1]$), 
then we get a random function $c \colon V(\hat{G}) \to M$. 
We will think of $c$ as a random $M$-coloring of the vertices of $\hat{G}$ 
that depends deterministically on the IID input labels. 
It is easy to see that this random coloring has the following properties. 
\begin{itemize}
\item The distribution of the random color of an $R$-nice vertex of type $v$ 
is $\mu_v^X$. Similarly, for an $R$-nice edge $\hat{e}$ the joint distribution 
of the colors on the endpoints of $\hat{e}$ is $\mu_e^X$ for the corresponding $e \in E(G)$. 
(See Theorem \ref{thm:general} for the definition of $\mu^X_v$ and $\mu^X_e$.) 
\item The color of a vertex depends only on the input labels in its $R$-neighborhood. 
That is, if we change the labels outside its $R$-neighborhood, its color remains the same. 
\end{itemize}

From now on we will assume that all but at most $\eps N$ edges of $\hat{G}$ are $R$-nice. 
Definition \ref{def:stat} defines the local statistics $\mu_v^c$ and $\mu_e^c$ 
of a deterministic coloring $c \colon V(\hat{G}) \to M$. 
Here we have a random coloring $c$, therefore $\mu_v^c$ and $\mu_e^c$ 
are random measures depending on the input labels. 
Taking expectation (with respect to the input labels) 
we get the measures $\E \mu_v^c$ and $\E \mu_e^c$. 
We claim that $\E \mu_e^c$ is $\eps$-close to $\mu_e^X$ 
in total variation distance for each $e \in E(G)$. 
This follows from the fact that the color pair of an $R$-nice lift of $e$ 
has distribution $\mu_e^X$ and that at most $\eps N$ edges are not $R$-nice 
among the $N$ lifts of $e$. 

Our goal is to show the existence of a deterministic coloring $c \colon V(\hat{G}) \to M$ 
with the property that $\mu_e^c$ is $\eps$-close to $\mu_e^X$ for each $e \in E(G)$. 
At this point we have a random coloring for which this is true in expectation. 
We will use the following form of the Azuma--Hoeffding inequality 
to show that the local statistics of our random coloring 
are concentrated around their expectations. 
\begin{lemma}
Let $(\Omega^n,\nu^n)$ be a product probability space. 
For a Lipschitz continuous function $f \colon \Omega^n \to \IR$ 
with Lipschitz constant $K$ (w.r.t.~the Hamming distance on $\Omega^n$) we have 
\begin{equation} \label{eq:azuma} 
\nu^n \left( \left\{ \omega \in \Omega^n \, : \, \left| f(\omega) - \E f \right| > \lambda \right\} \right) 
\leq 2 \exp \left( \frac{- \lambda^2 }{2 K^2 n} \right) .
\end{equation}
\end{lemma}
We use this in the following setting: $\Omega = [0,1]$, $\nu$ is the uniform measure on $[0,1]$, 
and $n = |V(\hat{G})| = N |V(G)|$. We will apply \eqref{eq:azuma} to different functions $f$. 
Next we describe these functions. 

Our random coloring $c$ depends on 
the configuration $\omega \in \Omega^n \cong [0,1]^{V(\hat{G})}$ of the input labels. 
For a given edge $e=(v_1,v_2) \in E(G)$ and a given pair of colors $m_1,m_2 \in M$ 
let $f(\omega) \defeq N \mu_e^c( \{ (m_1,m_2) \})$, 
that is, $f$ is the number of lifts of $e=(v_1,v_2)$ 
with the first endpoint having color $m_1$ and the second endpoint having color $m_2$. 
Using the fact that the random color of a vertex depends only on the input labels 
in its $R$-neighborhood, it is easy to see that $f$ is Lipschitz continuous 
with $K = 2 d_{\max}^{R+1}$, where $d_{\max}$ is the maximum degree of the base graph $G$. 

Using \eqref{eq:azuma} with $\lambda = \eps N$ we get that the probability 
that $\mu_e^c( \{ (m_1,m_2) \})$ is not $\eps$-close to $\E \mu_e^c( \{ (m_1,m_2) \})$ 
is very small: at most 2 $\exp(-\eps^2 N )$. 
Recall that $\eps$ can denote any quantity $C/ \log N$ 
where $C$ might depend on $G, M, R$ but not on $N$. 

Using union bound for all $e$ and all pairs $(m_1,m_2)$ 
we get that for large enough $N$ it holds with positive probability  
that $\mu_e^c$ is $\eps$-close to $\E \mu_e^c$ for each $e \in E(G)$. 
We have already seen that $\E \mu_e^c$ is $\eps$-close to $\mu_e^X$, 
thus we have proved the following. 
\begin{lemma} \label{lem:coloring}
Suppose that all but at most $\eps N$ edges of $\hat{G}$ are $R$-nice 
and that $N$ is large enough. 
Then there exists a deterministic coloring $c \colon V(\hat{G}) \to M$ 
such that $\mu_e^c$ is $\eps$-close (say in total variation distance) 
to $\mu_e^X$ for each edge $e \in E(G)$.  
\end{lemma}

\subsection*{c) The expected number of good colorings}
Next we determine the expected number of colorings 
with prescribed local statistics on random lifts of a base graph. 
These local statistics need to be consistent in the following sense. 
\begin{definition} \label{def:consistent}
For a finite simple graph $G$ and a finite color set $M$ 
by a \emph{consistent collection of distributions} we mean the following: 
a probability distribution $\mu_v$ on $M$ for each $v \in V(G)$ 
and a probability distribution $\mu_e$ on $M \times M$ for each $e \in E(G)$ 
such that the marginals of $\mu_e$ for $e=(u,v)$ are $\mu_u$ and $\mu_v$. 
\end{definition}
\begin{lemma} \label{lem:count}
Let $\mu_v$, $v \in V(G)$, and $\mu_e$, $e \in E(G)$, 
be a consistent collection of distributions as in the definition above. 
Recall that $\hat{G}_N$ denotes the random $N$-fold lift of $G$. 
Then the following formula holds for the expectation (w.r.t.\ $\hat{G}_N$) 
of the number of colorings on $\hat{G}_N$ for which 
the edge-statistics coincide with $\mu_e$: 
\begin{multline} \label{eq:expected_number}
\E_{\hat{G}_N} \left| \left\{ c \colon V( \hat{G}_N ) \to M  : \mu_e^c = \mu_e 
\ \forall e \in E(G) \right\} \right| \\
= \exp \left(  N \left( 
\sum_{e \in E(G)} H(\mu_e) - \sum_{v\in V(G)} (\deg v - 1) H(\mu_v) + o(1) 
\right) \right) \mbox{ as } N \to \infty 
\end{multline}
provided that the probabilities occuring in the discrete distributions $\mu_e$ 
are rational numbers and $N$ is a common multiple of all the denominators 
(otherwise the number of such colorings is clearly $0$). 
\end{lemma}
To prove the above lemma we will adapt the arguments 
in \cite[Section 4]{invtree} for our more general setting. 

Given a discrete distribution $\mu$ on $M$ (set of colors) 
the multinomial coefficients describe the number of 
$M$-colorings of a finite set with color distribution $\mu$. 
Using the Stirling formula it is easy to derive an asymptotic formula 
as the number of elements $N$ goes to infinity: there are 
$$ \exp\left( N \left( H(\mu) + o(1) \right) \right) $$ 
ways to choose the colors of $N$ elements 
in a way that the number of elements with color $m \in M$ is $N \mu( \{m\} )$ 
(provided that these numbers are integers). 

We will also need the following statement which is a slight variant of \cite[Lemma 4.1]{invtree}. 
\begin{claim*}
Let $L_u$ and $L_v$ be disjoint sets of size $N$. 
Fix $M$-colorings of $L_u$ and $L_v$ with color distributions $\mu_u$ and $\mu_v$, respectively. 
Let $\mu_e$ be any distribution on $M \times M$ with marginals $\mu_u$ and $\mu_v$ 
and with the property that all probabilities occuring in $\mu_e$ are multiples of $1/N$. 
Then the probability that a uniform random perfect matching 
between $L_u$ and $L_v$ has color distribution $\mu_e$ is 
\begin{equation} \label{eq:prob}
\exp\left( N \left( H(\mu_e) - H(\mu_u) - H(\mu_v) + o(1) \right) \right) .
\end{equation}
(The color distribution of a matching is 
the distribution of the pair of colors 
on the endpoints of the edges.)
\end{claim*}
Before proving this claim we show how Lemma \ref{lem:count} follows. 
First we take disjoint sets $L_v$ of size $N$ for each $v \in V(G)$. 
Then we color each $L_v$ with statistics $\mu_v$. This can be done in 
\begin{equation} \label{eq:count_vertex_coloring}
\exp \left(  N \left( \sum_{v \in V(G)} H(\mu_v) + o(1) \right) \right) 
\end{equation}
different ways. Let us fix such a coloring $c \colon \cup_{v \in V(G) } L_v \to M$. 
To get a random lift of $G$ we need to choose 
a uniform random perfect matching between $L_u$ and $L_v$ 
independently for each edge $e=(u,v)$. 
The probability that this perfect matching has statistics $\mu_e$ 
(for any fixed coloring $c$) is given by the formula \eqref{eq:prob}. 
These probabilities are independent and consequently the probability 
that a fixed coloring $c$ is ``good'' for a random lift 
is the product of \eqref{eq:prob} with $e$ running through $E(G)$. 
To get the expected number of good colorings for a random lift 
we need to multiply this product by \eqref{eq:count_vertex_coloring}, 
and Lemma \ref{lem:count} follows. 

Finally we prove the claim. 
\begin{proof}[Proof of Claim] 
By a \emph{colored perfect matching} between $L_u$ and $L_v$ 
we mean a coloring of the vertices in $L_u \cup L_v$ 
and a perfect matching between $L_u$ and $L_v$. 
There are two different ways to count the number of 
colored perfect matchings with color distribution $\mu_e$: 
\begin{multline*}
( \# \mbox{all perfect matchings} ) \cdot \exp\left( N \left( H(\mu_e) + o(1) \right) \right) \\
= \underbrace{ \left( \# \mbox{vertex colorings} \right)}
_{\exp\left( N \left( H(\mu_u) + H(\mu_v) + o(1) \right) \right)} 
\cdot ( \# \mbox{good perfect matchings for any given vertex-coloring} ) .
\end{multline*}
The claim immediately follows from this equality. 
\end{proof}

\subsection*{Putting the ingredients together}
As we explained in Section \ref{sec:block_factors}, 
an arbitrary $\Ga$-factor of IID process $X$ 
is the weak limit of finite-radius factors. 
Since the entropies $H(\mu^X_v)$ and $H(\mu^X_e)$ are continuous 
under weak convergence, it suffices to prove Theorem \ref{thm:general} 
for finite-radius factors. So let us assume that 
$X$ is a $\Ga$-factor of IID process with some finite radius $R$. 

On a random $N$-fold lift of $G$ let us consider the colorings $c$ 
with the property that $\mu_e^c$ is $\eps$-close to $\mu^X_e$ for all $e \in E(G)$. 
We claim that the expected number of such colorings on a random lift is, 
on the one hand, at least $1-o(1)$, and, on the other hand, 
asymptotically equal to 
\begin{equation} \label{eq:exp}
\exp \left(  N \left( 
\sum_{e \in E(G)} H(\mu^X_e) - \sum_{v\in V(G)} (\deg v - 1) H(\mu^X_v) + o(1) 
\right) \right) \mbox{ as } N \to \infty .
\end{equation}

Combining Lemma \ref{lem:large_girth} and Lemma \ref{lem:coloring} 
implies that at least one such coloring exists 
for a random $N$-fold lift of $G$ with probability $1-o(1)$. 
Therefore the expected number of such colorings is indeed at least $1-o(1)$. 

To get \eqref{eq:exp} we need to apply Lemma \ref{lem:count} 
for all collections of distributions $\mu_v$ and $\mu_e$ 
with the property that they are $\eps$-close to $\mu^X_v$ and $\mu^X_e$, respectively, 
and that all the probabilities occuring are multiples of $1/N$. 
It is easy to see that the total number of such collections is polynomial in $N$. 
We need to take the sum of \eqref{eq:expected_number} for all these collections. 
We can replace the entropies $H(\mu_v)$ and $H(\mu_e)$ with $H(\mu^X_v)$ and $H(\mu^X_e)$ 
at the expense of an $o(1)$ difference as $N \to \infty$. 
We get \eqref{eq:exp} with an extra factor that is polynomial in N 
but that can be also incorporated in the $N \cdot o(1)$ term in the exponent. 

Therefore \eqref{eq:exp} is at least $1-o(1)$ as $N \to \infty$ meaning that the term 
$$ \sum_{e \in E(G)} H(\mu^X_e) - \sum_{v\in V(G)} (\deg v - 1) H(\mu^X_v) $$ 
in the exponent cannot be negative, and this is exactly what we wanted to prove.

%%%%%%%%%%%%%%%%%%%%%%%%%%%%%%%%%%%%%%%%%%%%%%%%%%%%%%%%%%%%%%%%%%%%%%%%%
%%%%%%%%%%%%%%%%%%%%%%%%%%%%%%%%%%%%%%%%%%%%%%%%%%%%%%%%%%%%%%%%%%%%%%%%%
%%%%%%%%%%%%%%%%%%%%%%%%%%%%%%%%%%%%%%%%%%%%%%%%%%%%%%%%%%%%%%%%%%%%%%%%%

%\section{$G$-typical processes} \label{sec:typical}

%%%%%%%%%%%%%%%%%%%%%%%%%%%%%%%%%%%%%%%%%%%%%%%%%%%%%%%%%%%%%%%%%%%%%%%%%
%%%%%%%%%%%%%%%%%%%%%%%%%%%%%%%%%%%%%%%%%%%%%%%%%%%%%%%%%%%%%%%%%%%%%%%%%
%%%%%%%%%%%%%%%%%%%%%%%%%%%%%%%%%%%%%%%%%%%%%%%%%%%%%%%%%%%%%%%%%%%%%%%%%

\bibliographystyle{plain}
\bibliography{refs}

\end{document}